\documentclass{amsart}

\usepackage[T1]{fontenc}
\usepackage{amsmath,amssymb}

\newcommand{\C}{\mathbb C}
\newcommand{\D}{\mathbb D}
\newtheorem{theorem}{Theorem}[section]
\newtheorem{lemma}[theorem]{Lemma}
\newtheorem{proposition}[theorem]{Proposition}
\theoremstyle{remark}
\newtheorem{remark}[theorem]{Remark}

\title[A two-dimensional counterexample]{A two-dimensional counterexample for Kobayashi completeness of balanced domains}
\author{Armen Edigarian}
\address{Faculty of Mathematics and Computer Science, Jagiellonian University, Prof. St. \L{}ojasiewicza 6, 30-348 Krak\'ow, Poland}
\email{armen.edigarian@uj.edu.pl}
\date{September 5, 2026}

\begin{document}

\begin{abstract}
Jarnicki and Pflug constructed a bounded balanced domain of holomorphy in $\C^n$, $n\ge 3$, with continuous Minkowski function which is not Kobayashi complete, and explicitly asked whether the same phenomenon can occur in dimension two. We give such a domain in $\C^2$.
\end{abstract}

\maketitle

\section{Introduction}

A domain $G\subset\C^n$ is balanced if $\lambda G\subset G$ whenever $|\lambda|\le 1$. Its Minkowski function is
$h_G(z)=\inf\{t>0:z/t\in G\}$, so that $G=\{h_G<1\}$ and $h_G(\lambda z)=|\lambda|h_G(z)$. For bounded pseudoconvex balanced domains, completeness of invariant distances is closely related to regularity of the Minkowski function.

Barth proved that continuity of the Minkowski function is necessary for Kobayashi completeness of a bounded pseudoconvex balanced domain; see~\cite{Barth}. Jarnicki and Pflug showed that the converse fails in every dimension $n\ge3$: there exists a bounded balanced domain of holomorphy with continuous Minkowski function which is not Kobayashi complete~\cite{JPcounterexample}. Their construction, based on an idea of Sibony, raises the dimension by one. In their companion paper they formulated explicitly the remaining question as Problem~2.4: does their theorem remain valid for $n=2$? See~\cite[Problem~2.4]{JPproduct}. The same two-dimensional gap is recorded in the discussion of Kobayashi completeness in their monograph; see~\cite[Chapter~VII, \S7.5]{JPbook1} and the second edition~\cite[Chapter~14]{JPbook2}.

The purpose of this note is to answer this question affirmatively.

\begin{theorem}\label{thm:main}
There exists a bounded balanced pseudoconvex domain $G\subset\C^2$ with continuous Minkowski function such that $(G,k_G)$ is not complete. In particular, $G$ is a bounded balanced domain of holomorphy in $\C^2$ with continuous Minkowski function which is not Kobayashi-finitely-compact.
\end{theorem}

The point of the construction is that in dimension two a homogeneous plurisubharmonic function can be produced from a one-variable subharmonic function $u$ by setting $h(z,w)=|w|e^{u(z/w)}$. The difficulty is to arrange simultaneously two competing properties near a point $a_j\to0$: the value $u(a_j)$ must be sufficiently negative, while a disk much larger than $|a_j|$ must lie in a very small positive sublevel set of $u$. A logarithmic well with tiny mass and extremely small core radius decouples these two requirements.

Throughout, $p_{\D}$ denotes the Poincar\'e distance on $\D$ in the normalization
$p_{\D}(\alpha,\beta)=\operatorname{arctanh}\left|\frac{\alpha-\beta}{1-\overline\alpha\beta}\right|$. We use only its distance-decreasing property under holomorphic maps. Standard facts concerning plurisubharmonic functions and pseudoconvexity may be found, for example, in~\cite{Hormander}.

\section{The one-variable potential}

We begin with four sequences. For $j\ge1$ put $a_j=4^{-j}$, $R_j=2^{-j}$, $m_j=16^{-j}$ and $d_j=2^{-j/2}$. Thus $R_j=\sqrt{a_j}$, $m_j=a_j^2$, and $d_j\gg R_j\gg a_j$. Put $\tau_j=\exp(-d_j/m_j)$. Then $0<\tau_j<1$ and $m_j\log\tau_j=-d_j$. Define
\begin{equation}\label{eq:psi}
 \psi_j(\zeta)=\frac{m_j}{2}\log\frac{|\zeta/a_j-1|^2+\tau_j^2}{1+\tau_j^2},
 \qquad \zeta\in\C.
\end{equation}
Each $\psi_j$ is continuous and subharmonic, and $\psi_j(0)=0$. Its minimum is attained at $a_j$, where
\begin{equation}\label{eq:psi-center}
 \psi_j(a_j)=m_j\log\tau_j-\frac{m_j}{2}\log(1+\tau_j^2)
 \le m_j\log\tau_j=-d_j.
\end{equation}
Define
\begin{equation}\label{eq:u}
 u(\zeta)=\frac7{15}\log(1+|\zeta|^2)+\sum_{j=1}^{\infty}\psi_j(\zeta).
\end{equation}

\begin{lemma}\label{lem:convergence}
The series in \eqref{eq:u} converges locally uniformly on $\C$. Consequently $u$ is a continuous subharmonic function and $u(0)=0$.
\end{lemma}

\begin{proof}
Since $0<\tau_j<1$, formulae \eqref{eq:psi} and \eqref{eq:psi-center} give $\psi_j\ge-d_j-Cm_j$. On a fixed compact set $K$, we have $|\zeta/a_j-1|\le C_K/a_j$, hence $\psi_j\le C_Km_j(1+j)$. Therefore $|\psi_j|\le d_j+C_Km_j(1+j)$ on $K$. Both $\sum d_j$ and $\sum jm_j$ converge, so the Weierstrass test gives local uniform convergence. Local uniform limits of continuous subharmonic functions are continuous and subharmonic. Finally every summand vanishes at $0$, and hence $u(0)=0$.
\end{proof}

The wells are deep at their centers.

\begin{lemma}\label{lem:deep}
There is $j_0$ such that, for all $j\ge j_0$,
\begin{equation}\label{eq:deep}
 u(a_j)\le-\frac12d_j.
\end{equation}
\end{lemma}

\begin{proof}
If $k<j$, then $0<a_j<a_k$ and $|a_j-a_k|<a_k$, so $\psi_k(a_j)<0$. The $k=j$ term is at most $-d_j$ by \eqref{eq:psi-center}. If $k>j$, then $a_k<a_j$ and, since $\tau_k<1$,
$\psi_k(a_j)\le \frac{m_k}{2}\log((a_j/a_k+1)^2+1)\le m_k\log(3a_j/a_k)$.
Since $a_j/a_k=4^{k-j}$,
\[
 \sum_{k>j}\psi_k(a_j)\le C\sum_{k>j}16^{-k}(k-j+1)=O(16^{-j}).
\]
The first term in \eqref{eq:u} at $a_j$ is also $O(a_j^2)=O(16^{-j})$. Therefore $u(a_j)\le-d_j+C16^{-j}$, and \eqref{eq:deep} follows because $16^{-j}=o(d_j)$.
\end{proof}

The next estimate is the central point of the construction. 
\begin{lemma}\label{lem:sublevel}
For every $\zeta\in\mathbb C$ and every $j\ge1$,
\[
 \psi_j(\zeta)\le
 \frac{m_j}{a_j(1+\tau_j^2)}\,|\zeta|.
\]
Consequently,
$u(\zeta)\le \frac45|\zeta|$ for any $\zeta\in\mathbb C$.
\end{lemma}

\begin{proof} Note that
$\log(1+x+\frac{x^2}{2})\le x\quad\text{ for any }x\ge0$. Since $0<\tau_j<1$, we have
\[
 \psi_j(\zeta)
 \le \frac{m_j}{2}
 \log\frac{(1+\frac{|\zeta|}{a_j})^2+\tau_j^2}{1+\tau_j^2}\le \frac{m_j}{a_j(1+\tau_j^2)}\,|\zeta|.
\]
We now estimate $u$. We have  
$\log(1+|\zeta|^2)\le |\zeta|$ for any $\zeta\in\C$.
Using the estimate for each $\psi_j$ we obtain
$$ 
u(\zeta)
 \le \left(
 \frac{7}{15}
 +\sum_{j=1}^{\infty}
 \frac{m_j}{a_j(1+\tau_j^2)}
 \right)|\zeta|\le \left(
 \frac{7}{15}
 +\sum_{j=1}^{\infty}\frac{m_j}{a_j}
 \right)|\zeta|.
$$
We have $\sum_{j=1}^{\infty}\frac{m_j}{a_j}
 =\sum_{j=1}^{\infty}4^{-j}=\frac13$.
Consequently
\[
 u(\zeta)\le
 \left(\frac{7}{15}+\frac13\right)|\zeta|
 =\frac45|\zeta|.
\]
\end{proof}

We also need the precise logarithmic growth of $u$.

\begin{lemma}\label{lem:growth}
There exists a constant $C\in\mathbb R$ such that
\begin{equation}\label{eq:growth}
 u(\zeta)=\log|\zeta|+C+o(1)\qquad(|\zeta|\to\infty),
\end{equation}
uniformly with respect to the argument of $\zeta$.
\end{lemma}

\begin{proof}
For $|\zeta|\to\infty$, formula \eqref{eq:psi} gives
$\psi_j(\zeta)=m_j\log|\zeta|-m_j\log a_j-\frac{m_j}{2}\log(1+\tau_j^2)+r_j(\zeta)$,
where
\[
 r_j(\zeta)=\frac{m_j}{2}\log\left(|1-a_j/\zeta|^2+\tau_j^2a_j^2/|\zeta|^2\right).
\]
For $|\zeta|\ge1$ and all sufficiently large $j$, $|r_j(\zeta)|\le Cm_ja_j/|\zeta|$; the finitely many remaining indices satisfy the same estimate after increasing $C$. Hence $\sum_jr_j(\zeta)\to0$ uniformly as $|\zeta|\to\infty$. Also $\sum_jm_j|\log a_j|<\infty$ and $0\le\log(1+\tau_j^2)\le\log2$. Put $M=\sum_jm_j=1/15$. The first term in \eqref{eq:u} equals $(1-M)\log|\zeta|+o(1)$, while the logarithmic coefficient contributed by the wells is $M$. Hence the total coefficient of $\log|\zeta|$ is one, giving \eqref{eq:growth}.
\end{proof}

\section{Homogeneous lifting to $\C^2$}

Define, for $w\ne0$,
\begin{equation}\label{eq:h}
 h(z,w)=|w|e^{u(z/w)}.
\end{equation}
By Lemma~\ref{lem:growth}, $h(z,w)\to e^C|z|$ as $w\to0$ with $z\ne0$. We therefore put $h(z,0)=e^C|z|$ and $h(0,0)=0$.

\begin{lemma}\label{lem:h}
The function $h:\C^2\to[0,\infty)$ is continuous, positive on $\C^2\setminus\{0\}$, plurisubharmonic, and homogeneous: $h(\lambda z,\lambda w)=|\lambda|h(z,w)$ for all $\lambda\in\C$.
\end{lemma}

\begin{proof}
Continuity away from $w=0$ is clear. Lemma~\ref{lem:growth} gives continuity at every $(z,0)$ with $z\ne0$. Homogeneity then implies continuity at the origin: since $h$ is continuous on the unit sphere, $h(Z)\le C_1\|Z\|$ for all $Z$ near $0$. Positivity away from $0$ is immediate from the definition and from $h(z,0)=e^C|z|$.

On $\{w\ne0\}$, put $v=\log h$. The choice $M=\sum_jm_j$ makes the $\log|w|$ terms cancel exactly, and one obtains
\begin{equation}\label{eq:v}
\begin{split}
 v(z,w)={}&\frac{1-M}{2}\log(|z|^2+|w|^2)\\
 &+\sum_{j=1}^{\infty}\frac{m_j}{2}
 \log\frac{|z/a_j-w|^2+\tau_j^2|w|^2}{1+\tau_j^2}.
\end{split}
\end{equation}
Away from the origin the series in \eqref{eq:v} converges locally uniformly. Indeed, on a compact $K\Subset\C^2\setminus\{0\}$ choose $\eta>0$ so that $|z|+|w|\ge\eta$ on $K$. If $|z|\ge\eta/2$, then for all large $j$ the $j$th logarithm is bounded below by $-Cm_j$ and above by $C_Km_j(1+j)$. If $|w|\ge\eta/2$, its lower bound is $-d_j-C_Km_j$ and the same upper bound holds. Thus the series is uniformly dominated on $K$ by a summable sequence, since $\sum d_j<\infty$ and $\sum jm_j<\infty$. Every partial sum is plurisubharmonic, because each logarithm is the logarithm of a positive semidefinite Hermitian quadratic form. Hence \eqref{eq:v} defines a plurisubharmonic function on $\C^2\setminus\{0\}$.

At the origin, homogeneity and boundedness of $h$ on the unit sphere give $v(Z)\le\log\|Z\|+O(1)$, so $v$ extends plurisubharmonically by $v(0,0)=-\infty$. Formula \eqref{eq:v} also shows directly that on $w=0$, $v(z,0)=\log|z|+C$, in agreement with Lemma~\ref{lem:growth}. Finally $h=e^v$ is plurisubharmonic because the exponential is convex and increasing. Homogeneity follows directly from \eqref{eq:h} and the extension to $w=0$.
\end{proof}

\begin{proposition}\label{prop:G}
The set
\begin{equation}\label{eq:G}
 G=\{(z,w)\in\C^2:h(z,w)<1\}
\end{equation}
is a bounded balanced pseudoconvex domain, and $h$ is its continuous Minkowski function.
\end{proposition}

\begin{proof}
By continuity and homogeneity, $G$ is open and balanced, and it contains the origin; hence it is connected. Since $h$ is positive on the unit sphere, compactness gives $c:=\min_{\|Z\|=1}h(Z)>0$. Thus $h(Z)\ge c\|Z\|$, so $G\subset B(0,c^{-1})$ and is bounded. Since $\log h$ is plurisubharmonic on $\C^2$ in the extended sense, the standard criterion for balanced domains gives that $G=\{h<1\}$ is pseudoconvex. Equivalently, $-\log(1-h)$ is a continuous plurisubharmonic exhaustion of $G$. Homogeneity gives directly $h_G=h$. In particular, by the solution of the Levi problem, $G$ is a domain of holomorphy.
\end{proof}

\section{A finite Kobayashi-length chain}

We now show that $G$ is not Kobayashi complete. Discard finitely many initial indices and assume from now on that Lemmas~\ref{lem:deep} and~\ref{lem:sublevel} hold. Put $P_j=(a_j,1)$. Since $h(P_j)=e^{u(a_j)}<1$, we have $P_j\in G$, whereas $P_j\to P=(0,1)$ and $h(P)=e^{u(0)}=1$. Thus $P\in\partial G$.

Write $s_j=-u(a_j)$, so $s_j\ge d_j/2$. Set $t_j=\frac 85R_j$ and $r_j=e^{-t_j}$. Then $B(0,R_j)\subset\{u<t_j\}$. The proof uses two elementary geometric estimates.

\begin{lemma}\label{lem:radial}
Let $a\in\C$, put $s=-u(a)>0$, and let $0<t\le s$ and $r=e^{-t}$. Then $k_G((a,1),r(a,1))\le t/s$.
\end{lemma}

\begin{proof}
The map $\lambda\mapsto\lambda(a,1)$ sends $e^s\D$ into $G$, because $h(\lambda a,\lambda)=|\lambda|e^{-s}$. Hence $k_G((a,1),r(a,1))\le p_{\D}(e^{-s},e^{-(s+t)})$. The corresponding pseudohyperbolic distance satisfies
\[
 q=\frac{e^{-s}(1-e^{-t})}{1-e^{-(2s+t)}}
 \le\frac{e^{-s}t}{1-e^{-2s}}
 =\frac{t}{e^s-e^{-s}}\le\frac{t}{2s}.
\]
Since $t\le s$, we have $q\le1/2$, and therefore $\operatorname{arctanh}q\le2q\le t/s$.
\end{proof}

\begin{lemma}\label{lem:horizontal}
Suppose that $B(0,R)\subset\{u<t\}$, put $r=e^{-t}$, and let $|a|,|b|\le R/2$. Then $k_G(r(a,1),r(b,1))\le C|a-b|/R$, where $C$ is an absolute constant.
\end{lemma}

\begin{proof}
The map $F:R\D\to\C^2$, $F(\zeta)=r(\zeta,1)$, takes values in $G$, because $h(F(\zeta))=re^{u(\zeta)}<e^{-t}e^t=1$. Therefore $k_G(r(a,1),r(b,1))\le p_{R\D}(a,b)=p_{\D}(a/R,b/R)$. If $|x|,|y|\le1/2$, their pseudohyperbolic distance is $|x-y|/|1-\overline{x}y|\le\frac43|x-y|$ and is uniformly bounded away from $1$. Hence $p_{\D}(x,y)\le C|x-y|$, which gives the assertion.
\end{proof}

\begin{proof}[Proof of Theorem~\ref{thm:main}]
For all sufficiently large $j$, we have $t_j\le s_j$ and $t_j\le s_{j+1}$. Lemma~\ref{lem:radial}, first with $a=a_j$ and then with $a=a_{j+1}$, gives $k_G(P_j,r_jP_j)\le t_j/s_j\le C R_j/d_j$ and $k_G(r_jP_{j+1},P_{j+1})\le t_j/s_{j+1}\le C R_j/d_{j+1}$. Moreover, $a_j/R_j=2^{-j}$ and $a_{j+1}/R_j=2^{-j-2}$, so Lemma~\ref{lem:horizontal} gives $k_G(r_jP_j,r_jP_{j+1})\le C|a_j-a_{j+1}|/R_j$. By the triangle inequality,
\[
 k_G(P_j,P_{j+1})\le C\left(
 \frac{R_j}{d_j}+\frac{|a_j-a_{j+1}|}{R_j}+\frac{R_j}{d_{j+1}}
 \right).
\]
Here $R_j/d_j=2^{-j/2}$, $R_j/d_{j+1}=\sqrt2\,2^{-j/2}$, and $|a_j-a_{j+1}|/R_j\le2^{-j}$. Hence $\sum_jk_G(P_j,P_{j+1})<\infty$. It follows that $(P_j)$ is a $k_G$-Cauchy sequence. But $P_j\to(0,1)\in\partial G$, so it has no limit in $G$. Thus $(G,k_G)$ is not complete.

Because $G$ is bounded, it is Kobayashi hyperbolic. For the Kobayashi distance on a bounded domain, completeness is equivalent to Kobayashi finite compactness; see, for example,~\cite{JPbook1,JPbook2}. Therefore $G$ is not Kobayashi-finitely-compact either. Proposition~\ref{prop:G} gives all remaining assertions.
\end{proof}

\section{Remarks on the construction}

\begin{remark}
The original construction of Jarnicki and Pflug~\cite{JPcounterexample} begins with a noncomplete domain in $\C^2$ and homogenizes it to a balanced domain in $\C^3$. Lowering that argument by one dimension would require a corresponding one-dimensional noncomplete hyperbolic domain, which seems to be impossible.
\end{remark}

\begin{remark}
The particular sequences are not essential. What is used is the existence of $a_j\to0$, radii $R_j\gg a_j$, masses $m_j$, and depths $d_j$ such that the logarithmic wells have depth at least $d_j$, their positive influence on $B(0,R_j)$ is $O(R_j)$, and both $\sum R_j/d_j$ and $\sum |a_j-a_{j+1}|/R_j$ converge. The choices $a_j=4^{-j}$, $R_j=2^{-j}$, $m_j=16^{-j}$ and $d_j=2^{-j/2}$ make all estimates transparent.
\end{remark}

\medskip
\noindent\textit{Computational assistance.} ChatGPT was used in the exploration, verification, and editing of this manuscript. The author verified the mathematical arguments and takes responsibility for the final text.

\end{document}